\documentclass[11pt]{article}

\usepackage{geometry}                
\usepackage{graphicx}
\usepackage{amssymb}
\usepackage{epstopdf}
\usepackage{amsmath}
\usepackage{amscd}
\usepackage{amsthm}
\usepackage{mathrsfs}
\usepackage{amscd}
\usepackage{color}
\usepackage[T1]{fontenc}
\usepackage{CJKutf8}
\usepackage[english]{babel}
\usepackage{tikz,tikz-3dplot}
\usepackage{bm}
\usepackage{xcolor}
\usepackage[normalem]{ulem}
\usepackage[numbers,sort&compress]{natbib}
\usepackage{authblk}
\usepackage{float}

\newtheorem{thm}{Theorem}
\newtheorem{prop}[thm]{Proposition}
\newtheorem{cor}[thm]{Corollary}

\newtheorem{dfn}{Definition}
\newtheorem{example}{Example}

\newcommand{\Lfin}{L_{\mathrm{fin}}}

\newcommand{\Pf}{\mathcal{P}_{\hspace{-0.2mm}\mathrm{fin}}   }
\newcommand{\Mf}{\mathcal{M}_{\hspace{-0.2mm}\mathrm{fin}}   }

\renewcommand{\Re}{\mathbb{R}}

\title{Lattice Diversities}

\author[$\dagger$]{David Bryant}
\author[$\star$]{Ra\'ul Felipe}
\author[$\ddag$]{Mauricio Toledo-Acosta}
\author[$\odot$]{Paul Tupper}

\affil[$\dagger$]
{Department of Mathematics and Statistics, University of Otago, Dunedin, New Zealand. Email~\texttt{david.bryant@otago.ac.nz}}
\affil[$\star$]{CIMAT, Jalisco S/N, Valenciana. CP: 36240 Guanajuato, Gto, M\'exico. Email~\texttt{raulf@cimat.mx}}
\affil[$\ddag$]{Centro de Investigaci\'on en Ciencias, Universidad Aut\'onoma del Estado de Morelos, Cuernavaca, M\'exico. Email~\texttt{mauricio.toledo@uaem.mx}}
\affil[$\odot$]{Department of Mathematics, Simon Fraser University, Burnaby, Canada. Email~\texttt{pft@sfu.ca}}

\begin{document}

\maketitle

\begin{abstract}
Diversities are a generalization of metric spaces, where instead of the non-negative function being defined on pairs of points, it is defined on arbitrary finite sets of points. Diversities have a well-developed theory.
This includes the concept of a diversity tight span that extends the metric tight span in a natural way. 
Here we explore the generalization of diversities to lattices. Instead of defining diversities on finite subsets of a set we consider diversities defined on members of an arbitrary lattice (with a 0). 
We show that many of the basic properties of diversities continue to hold. However, the natural map from a lattice diversity to its tight span is not a lattice homomorphism, preventing the development of a complete tight span theory as in the metric and diversity cases.
\end{abstract}


\section{Introduction}

Let $\Pf(X)$ be the set of all finite subsets of a set $X$, and $\Re_{\geq 0}$ be the non-negative reals.
A diversity is a set $X$ with a function $\delta \colon \Pf(X)\rightarrow \Re_{\geq 0}$ satisfying the axioms
\begin{enumerate}
\item $\delta(A)=0$ if and only if $|A|\leq 1$.
\item For all $A,B,C \in \Pf(X)$ such that $B\neq\emptyset$
$$\delta(A\cup C)\leq \delta(A\cup B)+\delta(B\cup C).$$
\end{enumerate}
Diversities were introduced in \citep{Bryant2012} as a generalization of metric spaces. Indeed, if we let $d(a,b)=\delta(\{a,b\})$ for all $a,b \in X$, then $(X,d)$ is metric space.
Since their introduction a large body of theory has developed around the basic definition: connections to hypergraphs and combinatorial optimization \citep{BryantTupper2014}, hyperconvex geometry \citep{Piatek2014a,Espinola14}, and model theory \citep{bryant2020}, among others. 
In what follows we refer to diversities as defined in \citep{Bryant2012} to be \emph{classical diversities} in order to distinguish them from the concept of lattice diversities developed here.

The second condition in the definition of a classical diversity is known as the \emph{triangle inequality} and is equivalent to the following two conditions:
\begin{itemize}
\item If $A \subseteq B$, $\delta(A) \leq \delta(B)$, (\emph{monotonicity}) and 
\item If $A \cap B \neq \emptyset$ then $\delta(A \cup B) \leq \delta(A) + \delta(B)$, (\emph{subadditivity on intersecting sets})
\end{itemize}
for all $A,B \in \Pf(X)$.

One of the most attractive aspects of  these developments is the tight span theory for classical diversities, developed in \cite{Bryant2012} and generalizing the tight span theory for metric spaces developed in \cite{dress1984,isbell1964}. Given a classical diversity $(X,\delta)$, consider $P_X$ the set of all functions $f \colon \Pf(X) \rightarrow \mathbb{R}^+$ satisfying
\[
\sum_{i=1}^n f(A_i) \geq \delta( \cup_{i=1}^n A_i)
\]
for all finite subsets $A_i, i=1,\ldots,n$ of $X$.
We define the tight span $T_X$ of $X$ to be the set of minimal elements of $P_X$ under the standard ordering. It turns out that a function $\delta_T$ on $\Pf(T_X)$ can be defined so that 
$(T_X,\delta_T)$ is itself a classical diversity. Furthermore, there is a natural classical diversity embedding of $X$  into $T_X$, and $T_X$ is the minimal injective object that $X$ can be embedded in. 

The many developments of the theory of classical diversities motivates seeking generalizations.
Diversity functions are defined on the set of finite subsets of $X$, and the set of finite subsets of a set form a lattice. A natural question is to ask if the theory can be extended to more general lattices.

Here we see how far we can go by generalizing the axioms of classical diversities to arbitrary lattices containing a $0$. In Section~\ref{sec:basicdefn} we quickly review the necessary background of posets and lattices and give the basic definitions and properties of lattice diversities. In Section \ref{sec:always_diversities} we show 
 that a diversity function can always be defined on any lattice with 0 by letting $\delta$ be 0 on 0 and any atoms and 1 on all other elements.
We also show that, if the lattice satisfies some additional conditions, we can define  more interesting lattice diversities in a couple of ways: i) by using the length of finite chains, and ii) by using a valuation defined on the lattice.  In Section \ref{sec:distributive} we explore the case of lattice diversities for distributive lattices, and show that in the finite case these are exactly the restrictions of classical diversities to algebras of sets.
In Section \ref{sec:tight_span}, we define the tight span of a lattice diversity and compute it for two examples. We show that, unlike  the metric and diversity cases, there is in general no natural embedding of a lattice diversity into its tight span. Finally in Section~\ref{sec:discussion} we discuss some prospects for developing a more satisfying tight span theory for lattice diversities.

\section{Basic Definitions and Properties}\label{sec:basicdefn}

In order to generalize the theory of diversities to more general lattices we need to decide what will be the analogue of the empty set $\emptyset$ and singleton sets. Here we review the relevant concepts in order theory. See \cite{davey2002introduction} for a comprehensive introduction.

First we consider a partially ordered set (poset) $(P, \leq)$. That is, $P\neq\emptyset$ is any set and $\leq$ satisfies for all $a,b,c \in P$
\begin{enumerate}\setlength{\itemsep}{-0.0in}
\item $a \leq a$
\item $a \leq b$ and $b \leq c$ implies $a \leq c$,
\item $a \leq b$ and $b \leq a$ implies $a=b$.
\end{enumerate}
If there is an element $x \in P$ such that $x \leq y$ for all $y \in P$, we let $x=0$ and say that $P$ has a $0$. 

A lattice is a poset $(L,\leq)$ in which every two elements have a unique supremum (also called a least upper bound) and a unique infimum (also called a greatest lower bound). There are two binary operations defined on $L$ called meet (designated by $\wedge$) and join (designated by $\vee$). They are given by the supremum and infimum respectively. For $B \subseteq L$ we let $\bigvee B$ denote $\bigvee_{b \in B} b$; likewise for $\bigwedge B$. Lattices, being posets, may also have a $0$ element, which satisfies $0 \wedge x=0$ and $0 \vee x=x$ for all $x \in L$.

Given a poset $P$, not necessarily a lattice, we say $X \subseteq P$ is a \emph{lower set} of $P$, and write $X \in \mathcal{O}(P)$, if $x \in X$ and $y \leq x$ implies $y \in X$. $\mathcal{O}(P)$ is a distributive lattice with a $0$.
Given an element $x\in P$ we denote by $\downarrow x$ the set of all $y$ such that $y \leq x$; given a set $A \subseteq P$ we denote by $\downarrow A$ the set of all $y$ such that $y \in \downarrow a$ for some $a \in A$.  Both $\downarrow a \in \mathcal{O}(P)$ for $a \in P$ and $\downarrow A \in \mathcal{O}(P)$ for $A \subseteq P$.

Let $L$ be a lattice. An element $a\in L$ is said to be \emph{join irreducible} if $a \neq 0$ and  $a=b\vee c$ implies $a=b$ or $a=c$. We denote by $J(L)$ the set of all join irreducible elements. For $a,b\in L$, we say that $b$ covers $a$ (and that $b$ is an upper cover of $a$)
if $a < b$ and $\{c \mid a<c<b\}=\emptyset$. If $L$ has a least element $0$, then the upper covers of $0$ are called the atoms of $L$, and we denote them by $A(L)$.

In the lattice of finite subsets of a set $X$, $0$ is the empty set, and $A(L)$ is the set of singletons. This observation motivates the following definition. 
\begin{dfn} \label{defn:basic} Let $L$ be a lattice with $0$ and $\delta \colon L \rightarrow \mathbb{R}^+$. $(L,\delta)$ is a \emph{lattice diversity} if
\begin{enumerate}
\item $\delta(a)=0$ if and only if $a \in A(L)$ or $a=0$.
\item $a \leq b$ implies $\delta(a) \leq \delta(b)$ (monotonicity), 
\item $a \wedge b \neq 0$ implies $\delta(a \vee b) \leq \delta(a)+ \delta(b)$ (subadditivity on non-disjoint elements)
\end{enumerate}
\end{dfn}

We remark that the last two conditions of the definition when restricted to the lattice $\Pf(X)$ of a set $X$ reduce to monotonicity ($A \subseteq B$ implies $\delta(A)\leq \delta(B)$) and subadditivity on intersecting sets ($A \cap B \neq \emptyset$ implies $\delta(A \cup B) \leq \delta(A) + \delta(B)$). In the original theory of diversities these two conditions are equivalent to the triangle inequality,  but for general lattices
this equivalence does not hold. An example is given by a four-point lattice with points $0 < a_1 < a_2 < a_3$, and $\delta(0)=\delta(a_1)=0$, $\delta(a_2)=2$ and $\delta(a_3)=1$, which satisfies the triangle inequality but not monotonicity. We required montonicity in order to get a more usable theory.




The following are examples of lattice diversities:

\begin{example}
Observe that if we take the lattice $L$ to be $\Pf(X)$, the lattice of finite subsets of a set $X$, then we recover the original definition of classical diversities in \citep{Bryant2012}. In this case, the role of $X$ itself is played by $A(L)=J(L)$. Likewise, if we take any classical diversity $(X,\delta)$, and we restrict $\delta$ to a sublattice of $\Pf(X)$ we obtain a lattice diversity.
\end{example}

\begin{example}
Let $L$ be any totally ordered set with a least element $0$.  Let $\delta \colon L \rightarrow \mathbb{R}^+$ satisfy $\delta(a)=0$ if and only if $a \in \{0\} \cup A(L)$ and $a \leq b$ implies $\delta(a) \leq \delta(b)$. Then $(L,\delta)$ is a lattice diversity, since subadditivity on non-disjoint elements follows immediately. This example works for $L=\mathbb{N}$ with $A(L)=\{1\}$, as well as for $L=\mathbb{R}^+$, where $A(L)$ is empty.
\end{example}

\begin{example}Let $\Mf(X)$ denote the set of finite multisets of a set $X$. Multisets are subsets that allow the inclusion of multiple copies of elements of a set $X\neq \emptyset$. For $x_1,...,x_n\in X$  and $k_1,k_2,...,k_n \in \mathbb{Z}_{\geq 0}$, we denote by $\{x_1^{(k_1)},...,x_n^{(k_n)}\}$  an element of $\Mf(X)$ made up of $k_j$ copies of $x_j$. The join of multisets is obtained by taking the maximum number of each element in the two multisets, and the meet is obtained by taking the minimum.

We give one example of how to define a diversity on $\Mf(X)$. Let $d$ be a metric on $X$. For each $x \in X$ let the function $f_x \colon \mathbb{N} \rightarrow \mathbb{R}^+$ be monotone and  satisfy $f_x(n)=0$ if and only if $n=0$ or $1$. Then let
\[
\delta(\{x_1^{(k_1)},...,x_n^{(k_n)}\}) = \max \left( \max_{i,j} d(x_i,x_j), \max_i f_{x_i}(k_i) \right).
\]
Then $(\Mf(X),\delta)$ is a lattice diversity, and $\delta$ is the minimal diversity function that dominates the diameter diversity of $(X,d)$ and the single-chain lattice diversities given by the $f_{x_i}$.
\end{example}

\begin{example} \label{ex:divis} Let  $L$ be the set of all natural numbers, where  the meet operation $\wedge$ is the greatest common divisor and the join operation $\vee$ is the lowest common multiple. In this case, $n \leq m$ means that $n$ divides $m$. The least element of $L$ is $1$ and $A(L)$ is the set of prime numbers. We call this lattice the \emph{divisibility lattice}.  This example of a lattice is equivalent to the previous example, where the base set $X$ is the set of primes.


As an example of a diversity function on $L$, we derive one from the prime Omega function $\Omega(n)$ which counts the total number of prime factors of $n$, counting multiplicity (see \citep{hardy1979introduction}). For example,  $\Omega(1)=0$, $\Omega(7)=1$, $\Omega(6)=2$, and  $\Omega(16)=4$. Consider the function 
	$$\delta(n)=\begin{cases}
	\Omega(n), & n\text{ is not prime,}\\
	0, & n\text{ otherwise}.
	\end{cases}	
	$$  
It is straightforward to see that $\delta=0$ on $A(L)\cup\{1\}$ and that $\delta$ is monotone. The fact that $\delta$ is subadditive is a consequence of $\Omega$ being completely additive, i.e.\ $\Omega(mn)=\Omega(m) + \Omega(n)$. So $\delta$ is a diversity on $L$.\\
\end{example}

Now we can establish some basic properties.

\begin{prop}\label{prop:general_subadditivity}
Let $(L,\delta)$ be a lattice diversity and let $x_1,...,x_n\in L$ such that $x_1\wedge ... \wedge x_n\neq 0$. Then
	$$\delta(x_1\vee...\vee x_n)\leq \delta(x_1) +...+\delta(x_n).$$
\end{prop}

\begin{proof}
We proceed by induction. By the subadditivity of $\delta$  on non-disjoint elements the statement is true for $n=2$. Assume it is true for $n-1$ and that  $c :=x_1\wedge ... \wedge x_n\neq 0$.   We know $c \leq x_i$ for all $i$ and so in particuluar $c\leq x_n$ and $c \leq x_1 \vee \ldots \vee x_{n-1}$. Hence
\[
0 < c \leq (x_1 \vee \ldots \vee x_{n-1}) \wedge x_n.
\]
So we can use the inductive hypothesis and apply subadditivity again
to get
\[
\delta((x_1\vee...\vee x_{n-1}) \vee x_n) \leq \delta(x_1 \vee \ldots \vee x_{n-1}) + \delta(x_n) \leq \sum_{i=1}^{n-1} \delta(x_i) + \delta(x_n)
\]
as required.
 \end{proof}

\begin{thm}
For a lattice diversity $(L,\delta)$, the diversity function $\delta$ satisfies the \emph{triangle inequality}: 
	$$\delta(a \vee c) \leq \delta(a \vee b) + \delta(b \vee c),$$
for all $a,b,c \in L$ with $b \neq 0$. 
\end{thm}

\begin{proof}
Let $a,b,c\in L$, with $b\neq 0$. Then, by montonicity
	\begin{align*}
	\delta(a\vee c) &\leq \delta(a\vee b \vee c) = \delta\left( (a\vee b) \vee (b\vee c) \right)\\
		&\leq \delta(a\vee b) + \delta(b\vee c),
	\end{align*}
	since $(a \vee b) \wedge (b \vee c) \geq b \neq 0$.
  \end{proof}

\begin{thm}
For a lattice diversity $(L,\delta)$, define the function $d_\delta \colon A(L) \times A(L) \rightarrow \mathbb{R}^+$ by
\[
d_\delta(a,b) = \delta( a \vee b)
\]
then $(A(L),d_\delta)$ is a metric space.
\end{thm}

\begin{proof}
It is clear that $d_\delta$ is a non-negative, symmetric function satisfying $d_\delta(a,a)=0$ for any $a\in A(S)$. If $a,b\in A(L)$ are two elements such that $d_\delta(a,b)=0$, then $a\vee b=0$ or $a\vee b\in A(L)$. In the former case, it follows that $a=b=0$. In the latter case, $a\vee b$ covers $0$ and also $a\vee b\geq a,b$, which yields $a=b$.

Finally, let $a,b,c\in A(L)$. Then $c\neq 0$ and therefore,
	\[
	d_\delta(a,b)=\delta(a\vee b)\leq \delta(a\vee c)+\delta(c\vee b) = d_\delta(a,c) + d_\delta(c,b).
	\]
  \end{proof}

Beside generalizing metrics, lattice diversities can also be regarded as a generalization of some multiway metrics that have been defined in the literature. Here we show that  lattice diversities define an $n$-\emph{way distance} on $A(L)$. Deza and Rosenberg \cite{deza2000n}  define an $n$-way distance on a set $X$ to be a totally symmetric map $d \colon X^n \rightarrow \mathbb{R}^+$ such that, for all $x_1, \ldots, x_{n+1} \in X$ 
\begin{enumerate}
\item $d(x_1, \ldots, x_1)=0$,
\item $
d(x_1,\ldots, x_n) \leq \sum_{i=2}^{n+1} d(x_1,\ldots, x_{i-1}, x_{i+1}, \ldots, x_{n+1})
$
\item
$
d(x_1, x_1, x_3, \ldots,  x_n)= d(x_1, x_3, x_3, \ldots, x_n) \leq d(x_1,x_2,x_3,\ldots x_n).
$
\end{enumerate}

We note that, in the terminology of \cite{deza2000n}, every $n$-way distance is also a {\em weak $n$-way distance}, and therefore also an $(n-1)$-semi-metric.

\begin{prop}
Let $\delta$ be a diversity on a lattice $L$, 
and define $d_{n,\delta}:X^n \rightarrow \Re^{+}$ by
	$$d_{n,\delta}(a_1,...,a_n)=\delta\left(a_1\vee...\vee a_n \right).$$
Then $d_{n,\delta}$ is an $n$-way distance on $A(L)$. 
\end{prop}

\begin{proof}
We establish each property is turn.
\begin{enumerate}
\item $d_{n,\delta}(a_1,\ldots,a_1) = \delta(a_1 \vee \ldots \vee a_1)=\delta(a_1)=0$.
\item Let $b_i = a_1 \vee \ldots \vee a_{i-1} \vee a_{i+1} \vee \ldots \vee a_{n+1}$, for $i=2,\ldots, n+1$. So we have to show $\delta( \bigvee_{i=1}^n a_i) \leq \sum_{i=2}^{n-1} \delta(b_i)$. Note that all such $b_i$ satisfy $a_1 \leq b_i$ and  so
\[
\bigwedge_{i=2}^{n+1} b_i \geq a_1 >0.
\]
Also note by expanding the $b_i$ and using basic lattice identities that 
\[
\bigvee_{i=2}^{n+1} b_i = \bigvee_{i=1}^{n+1} a_i \geq  \bigvee_{i=1}^{n} a_i.
\]
Hence applying monotonicty and Proposition~\ref{prop:general_subadditivity}, we have 
\[
\delta\left( \bigvee_{i=1}^n a_i \right) \leq \delta\left(\bigvee_{i=2}^{n+1} b_i\right) \leq \sum_{i=2}^{n+1} \delta(b_i),
\]
which gives the inequality.
\item This follow from $d_{n,\delta}$ only depending on the set of values in its arguments, and not how many times they appear, and the fact that $\delta$ is monotonic.
\end{enumerate}
  \end{proof}




\section{General Lattice Diversity Constructions}
\label{sec:always_diversities}

Here we show that given a lattice $L$ with a $0$ we can always define a function $\delta \colon L \rightarrow \Re_{\geq 0}$ such that $(L,\delta)$ is a lattice diversity.
Then we show some other general constructions for more restrictive classes of lattices.

\begin{thm}
Let $L$ be any lattice with a $0$ and let $\delta:L\rightarrow \Re_{\geq 0}$ be the function given by 
	$$\delta(a)=\begin{cases} 0, & a\in A(L)\cup\{0\},\\
	1, & \text{otherwise.} \end{cases}$$
	Then $(L,\delta)$ is a lattice diversity.
\end{thm}
\begin{proof}
The first two conditions of Definition~\ref{defn:basic} are straightforward to verify.
For the third condition, let $a,b \in L$ with $a \wedge b \neq 0$. The only way the inequality can fail is if $\delta(a \wedge b)=1$ and $\delta(a)=\delta(b)=0$. 

So we suppose that $\delta(a)=\delta(b)=0$. Then both $a, b \in A(L) \cup {0}$. If either $a$ or $b$ is $0$, then we immediately get $\delta(a \wedge b)=0$. So suppose both $a,b \in A(L)$.
Since $a \wedge b \neq 0$, 
we must have $a=b=a\wedge b$ and so $\delta(a \wedge b)=0$. So the inequality always holds. 
  \end{proof}

A lattice diversity $(L,\delta)$ is said to be {\em strictly monotone} if $a<b$ implies $\delta(a) < \delta(b)$ for all $a,b \neq 0$. We now show that with a few additional restrictions on $L$ there always exists at least one diversity $(L,\delta)$ that is strictly monotone.


For $a,b\in L$ such that $a\leq b$, we denote $I[a,b]=\{c\in L | a\leq c\leq b\}.$ 
A lattice $L$ is modular if for any elements $a,b,c \in L$, if $c\leq a$, then $a\wedge (b\vee c)= (a\wedge b)\vee c$. A lattice $L$ is said to be of \emph{finite height} if there is an upper bound to the length of chains in $L$, where the length of a chain of $n+1$ elements is $n$. The height of $L$ is the least such upper bound. $L$ is said to be \emph{sectionally of finite height} iff $L$ has a $0$, and for every $a\in L$, the interval $I[0,a]$ is of finite height. In this case, the height of $I[0,a]$ will be denoted by $h(a)$ and called the \emph{height} of $a$ \cite{mckenzie2018algebras}. An interesting example of a height function is for the divisibility lattice (Example \ref{ex:divis}) for which we know $h=\Omega$ the prime Omega function. 

We note that all distributive lattices are modular. As well, all finite lattices are sectionally of finite height. Also, for any set $X$ the lattice $\Pf(X)$  is modular and sectionally of finite height. In this case, $h$ gives the cardinality of a  subset.

Height allows us to define a strictly monotone diversity on a very general class of lattices.
The following example generalizes the classical diversity on sets that assigns the value cardinality minus 1 to non-empty finite subsets.

\begin{thm}
 Let $L$ be a modular lattice sectionally of finite height and let $\delta_h:L\rightarrow \Re_{\geq 0}$ be given by
	$$\delta_h(a)=\begin{cases} 0, & a=0,\\
	h(a)-1, & \text{otherwise.} \end{cases}$$
Then $(L,\delta_h)$ is a strictly monotone lattice diversity.
\end{thm}

\begin{proof}
Theorem 2.39 of \citep{mckenzie2018algebras} states that under the conditions of the theorem, $h(0)=0$, $a<b$ implies $h(a)<h(b)$, and $h(a)+h(b) = h(a \vee b) + h(a \wedge b)$. The first
condition of being a lattice diversity follows since $h(a)=1$ if and only if $a \in A(L)$.
 The second condition of being a lattice diversity follows from the monotonicity of $h$.

To show the third condition, suppose $a \wedge b \neq 0$. This implies none of $a, b, a \vee b$ are $0$ either, and so
\begin{eqnarray*}
\delta_h(a \vee b) & = & h(a \vee b)-1 \\
& = & h(a)+ h(b) -h (a \wedge b) -1 \\
& = & [h(a)-1] + [h(b)-1] - [h(a \wedge b)-1] \\
&= & \delta_h(a) + \delta_h(b) - \delta_h(a \wedge b) 
\end{eqnarray*}
which is less than $\delta_h(a) + \delta_h(b)$ since $\delta_h(a \wedge b)\geq 0$, giving the desired inequality.
  \end{proof}

The classical diversity that assigns the cardinality of a set to sets of more than two points can be  generalized  with the following result.

\begin{cor}
 Let $L$ be a modular lattice sectionally of finite height and let $\delta_c:L\rightarrow \Re_{\geq 0}$ be given by
	$$\delta_c(a)=\begin{cases} 0, & a\in A(L)\cup\{0\},\\
	h(a), & \text{otherwise.} \end{cases}$$
Then $(L,\delta_c)$ is a lattice diversity.
\end{cor}

\begin{proof}
The proof that $\delta_c$ satisfies the first two diversity axioms follows exactly as in the previous theorem. To show the third condition, suppose $a \wedge b \neq 0$. If $a$ and $b$ are atoms we must have $a=b=a \vee b$ and so the inequality holds with $\delta_h(a)=\delta_h(b) = \delta_h(a \vee b)=0$. If $a$ is an atom and $b$ is not, then $a \leq b = a \vee b$, and so the inequality holds with $\delta_h(a)=0$ and $\delta_h(b)=\delta_h(a \vee b)$. Finally, if neither $a$ nor $b$ are atoms, neither is $a \vee b$ and so
$$\delta_h(a \vee b) = h(a \vee b)= h(a) + h(b) - h(a \wedge b) \leq h(a) + h(b) = \delta_h(a) + \delta_h(b).$$ 
  \end{proof}

The lattice diversity  we have defined in the last corollary using $h$ is a special case of lattices defined using valuations.
A \emph{valuation} on a lattice $L$ is a non-negative real-valued function on $L$ such that 
	$$v(a\wedge b ) + v(a\vee b) = v(a) + v(b).$$
A \emph{sub-valuation} on a lattice $L$ is a non-negative real valued function on $L$ such that 
	$$v(a\wedge b ) + v(a\vee b) \leq v(a) + v(b).$$
A valuation (resp. sub-valuation) is \emph{positive} if $v(a)<v(b)$ whenever $a<b$ (for more details on valuations, see \citep{birkhoff1967lattice} or \citep[Section 10.3]{deza2009encyclopedia}). The height function is an example of a positive valuation on modular lattices of sectionally finite height. Another example of a positive valuation is the function $\log(x)$ on the lattice of divisibility of positive integers.

The following theorem states that we can define diversities from positive sub-valuations on a lattice. 

\begin{thm}
Let $L$ be a lattice with a $0$ and let $v:L\rightarrow\Re_{\geq 0}$ be a positive sub-valuation with $v(0)=0$. Let $\delta_v:L\rightarrow \Re_{\geq 0}$ be given by
	$$\delta_v(a)=\begin{cases} 0, & a\in A(L),\\
	v(a), & \text{otherwise.} \end{cases}$$
Then, $(L,\delta_v)$ is a strictly monotone lattice diversity.
\end{thm}

\begin{proof}
If $a\in A(L)\cup\{0\}$, then $\delta_v(0)=0$. If $a\in L$ is such that $\delta_v(a)=0$, then $a\in A(L)$ or $v(a)=0$, the latter possibility implying that $a=0$, since $v$ is a positive sub-valuation.\\
Let $a,b\in L\setminus A(L)$ such that $a\leq b$. Then $v(a)\leq v(b)$ and therefore, $\delta_v(a)\leq \delta_v(b)$. If $a\in A(L)$, then $\delta_v(a)=0\leq \delta_v(b)$. The argument is similar if $b\in A(L)$.\\
\\
To check that $\delta_v$ is strictly monotone, let $0<a<b$. Either $a \in A(L)$, in which case $b \not\in A(L)$ and so $\delta_v(a)=0<v(b)=\delta_v(b)$
or $\delta_v(a)=v(a)<v(b)=\delta_v(b)$.\\
Since $v$ is a sub-valuation, we have that
	\begin{equation}\label{eq_valuation}
	v(a\vee b) = v(a)+v(b) - v(a\wedge b) \leq v(a)+v(b).
	\end{equation}
Let $a,b\in L$ such that $a\wedge b\neq 0$. Assume first that $a,b\not\in A(L)\cup\{0\}$, then (\ref{eq_valuation}) implies $\delta_v(a\vee b) \leq \delta_v(a)+\delta_v(b)$. If $a=0$ or $b=0$, then we cannot have $a\wedge b\neq 0$. If $a\in A(L)$, then $a\leq b$ (otherwise, if $a$ and $b$ are not comparable, we would have $a\wedge b = 0$) and then, $v(a\vee b)=v( b)\leq v(a)+v(b)$. 	 
  \end{proof}



\section{Finite Distributive Lattice Diversities} \label{sec:distributive}

In the classical theory of diversities introduced by \citep{Bryant2012}, diversity functions are defined on the lattice of finite subsets of a set, which is of course a distributive lattice. If we take the restriction of the diversity to a sublattice of this lattice, we obtain a lattice diversity in our framework here. This raises the question of whether all lattice diversities on distributive lattices can be obtained this way. Here we show that the answer is yes in the finite case, in that every  lattice diversity on a finite distributive lattice is isomorphic to a restriction of a classical diversity to a sublattice of the lattice of finite subsets.

We recall that Birkhoff's representation theorem \citep[Theorem 5.12]{davey2002introduction} states that every finite distributive lattice $L$ is isomorphic to the lattice of down-sets $\mathcal{O}\left(J(L)\right)$ of the join-irreducible elements $J(L)$ via the isomorphism
	\begin{equation} \label{eqn:isomorph}
	\eta(a)=J(L)\cap \downarrow a,
	\end{equation}
	whose inverse is given by
	\[
	\eta^{-1}(A) = \bigvee A,
	\]
	for $A \subseteq J(L)$ \cite[Prop 2.45]{davey2002introduction}.
Therefore, if $L$ is a finite distributive lattice with $0$, and $(L,\delta)$ is a lattice diversity, then it induces a lattice diversity $(\mathcal{O}\left(J(L)\right),\hat{\delta})$ in a canonical way, that is, $\hat{\delta}(A) = \delta( \eta^{-1}(A)) = \delta(\bigvee A)$. In the next theorem we show that this can be extended to a diversity on $\mathcal{P}\left(J(L)\right)$, the set of all subsets of $J(L)$.

\begin{thm}
Let $(L,\delta)$ be a lattice diversity where $L$ is finite distributive lattice.
For $A\subseteq J(L)$, we define
	$$\hat{\delta}(A)=\begin{cases}
	\delta \left( \bigvee A \right), & \text{if }\vert A \vert \geq 2 \\
	0, & \text{otherwise.}
	\end{cases}$$
Then $(J(L),\hat{\delta})$ is a classical diversity whose restriction to $\mathcal{O}\left(J(L)\right)$ is isomorphic to $(L,\delta)$.
\end{thm}

\begin{proof}
By definition $\hat{\delta}(A)=0$ if $\vert A\vert \leq 1$. If $A\subseteq J(L)$ is such that $\hat{\delta}(A)=0$, then, either $\vert A \vert\leq 1$ or $\delta(\bigvee A)=0$. In the latter case, the lattice diversity axioms imply that
$\bigvee A$ is either $0$ (in which case $A$ is empty) or an atom (in which case $A$ consists of just that atom). 
Together this implies that $\hat{\delta}(A)=0$ if and only if $\vert A\vert\leq 1$.

To check monotonicity, we observe that $A \subset B$ implies $\bigvee A \leq \bigvee B$ and then use the monotonicity of $\delta$.

To check subadditivity, let $A,B\subseteq J(L)$ such that $A\cap B\neq \emptyset$. Then $\bigvee A \wedge \bigvee B  \neq \emptyset$, and so 
\[
\hat{\delta}( A \cup B) =\delta\left( \bigvee (A \cup B)\right)= \delta\left(\bigvee A \vee \bigvee B\right) \leq \delta\left(\bigvee A\right) + \delta\left(\bigvee B\right) 
= \hat{\delta}(A) + \hat{\delta}(B),
\]
as required.
  \end{proof}

\section{The Tight Span}
\label{sec:tight_span}

Here we develop the tight span theory for lattice diversities, following closely the development for classical diversities in  \cite{Bryant2012}. 


\begin{dfn}\label{def:tight_span}
Let $(L,\delta)$ be a lattice diversity. Let $P_L$ denote the set of all functions $f:L\rightarrow \Re_{\geq 0}$ satisfying
    \begin{equation}\label{eq:defn_P}
    \sum_{b\in B}f(b)\geq \delta\left(\bigvee B\right),
    \end{equation}
for any finite subset $B\subset\Pf(L)$. Write $f\preceq g$ if $f(a)\leq g(a)$ $\forall a\in L$. The \emph{tight-span} of $(L,\delta)$ is the set $T_L\subseteq P_L$ of functions that are minimal under $\preceq$.
\end{dfn}

\begin{prop}\label{prop:zerolemma}
Let $(L,\delta)$ be a lattice diversity and 
let $f\in T_L$. Then $f(0)=0$.
\end{prop}

\begin{proof}
Let us suppose that $f(0)>0$. Let $g:L\rightarrow \Re_{\geq 0}$ be given by 
	$$g(a)=\begin{cases}
	f(a), & a\neq 0,\\
	0,& a=0.
	\end{cases}$$
Let $B\in\Pf(L)$. If $0\not\in B$, then 
	$$\sum_{b\in B}g(b)=\sum_{b\in B}f(b)\geq \delta\left(\bigvee B\right),$$
satisfying the inequality. If $0\in B$,
	$$\sum_{b\in B}g(b)=\sum_{b\in B\setminus\{0\}}f(b) + g(0)=\sum_{b\in B\setminus\{0\}}f(b)\geq \delta\left( \bigvee \left( B\setminus\{0\}\right)\right) = \delta\left(\bigvee B\right),$$
satisfying the inequality again. So $g \in P_L$. Now, observe that $g\preceq f$ and $g\neq f$. This contradicts the fact that $f$ is minimal. Hence $f(0)=0$. 
  \end{proof}

The following is analogous to Proposition 2.3 of \cite{Bryant2012}, and follows from the exact same argument  (replacing $\cup$ and $\cap$ with $\vee$ and $\wedge$, etc.)\ together with Proposition~\ref{prop:zerolemma} here.

\begin{prop}
Let $f \colon L \rightarrow \mathbb{R}^+$. Then $f \in T_L$ if and only if for all $a \in L$
\[
f(a) = \sup_{B \in \Pf(L)} \left\{ \delta\left( a \vee \bigvee B\right) - \sum_{b \in B} f(b) \right\}.
\]
\end{prop}

Analogously to Proposition 2.4 of \cite{Bryant2012}, we can prove that functions in $T_L$ satisfy the following properties.

\begin{prop}\label{prop:basic_properties_TP}
Suppose that $f\in T_L$. Then
	\begin{enumerate}
	\item $f(a)\geq \delta(a)$ for any $a\in L$.
	\item Let $a,b\in L$ such that $a\leq b$ then $f(a)\leq f(b)$. That is, $f$ is monotone.
	\item $f(a\vee c)\leq \delta(a\vee b)+ f(b\vee c)$ for any $a,b,c\in L)$, $b\neq 0$.
	\item $f(a\vee b)\leq f(a)+f(b)$ for any $a,b\in L$. That is, $f$ is sub-additive.
	\item For all $a \in L$, 
		$$f(a)=\sup_{b\in \Lfin(L)}\{\delta(a\vee b)- f(b)\}.$$
	\end{enumerate}
\end{prop}

Again the proof follows by the exact same techniques as Proposition 2.4 in \cite{Bryant2012}.

For $x\in L$, let $h_x:L\rightarrow \Re_{\geq 0}$ be the function given by
	$$h_x(a)=\delta\left( x\vee a\right).$$
We define the map $\kappa:L\rightarrow P_L$ given by $\kappa(x)= h_x$. In the following theorem we prove that, indeed, $h_x\in P_L$.
	
\begin{thm}
Let $(L,\delta)$ be a lattice diversity.
Denote by $\leq$ the order on $L$ defined by $\vee$. The function $\kappa:(L,\leq)\rightarrow \left(P_L \cup \{\delta\},\preceq\right)$ defined above is an order-preserving map. Furthermore, 
	\begin{enumerate}
	\item $\kappa(0)=\delta$.
	\item $\kappa(a)\in P_L$ if $a\neq 0$.
	\item $\kappa(a)\in T_L$ if $a\in A(L)$.
	\item $\kappa$ restricted to $A(L)\cup\{0\}$ is injective.
	\item Recall that $\delta$ is strictly monotone iff $a<b$ implies $\delta(a)<\delta(b)$. If $\delta$ is strictly monotone, then $\kappa$ is injective.
	\end{enumerate}
\end{thm}

\begin{proof}
Let $a,b\in L$ such that $a\leq b$. Then, for any $x\in L$, $h_a(x)=\delta(a\vee x)\leq \delta(b\vee x)=h_b(x)$ and hence $h_a\preceq h_b$. This verifies that $\kappa$ is an order-preserving mapping. Now, we prove each statement:

\begin{enumerate}
\item For any $x\in L$, $h_0(x)=\delta(x)$, therefore $\kappa(0)=\delta$.
\item Let $a\in L\setminus\{0\}$ and $B\in\Pf(L)$.  Then
	$$\sum_{b\in B}h_a(b)=\sum_{b\in B}\delta(a\vee b)\geq \delta\left(\bigvee B\right).$$
	This follows from Proposition \ref{prop:general_subadditivity} and the fact that, if we write $B=\{b_1,...,b_n\},$  
	$$(a\vee b_1)\wedge ... \wedge (a\vee b_n) \geq a \neq 0.$$ 
	Therefore $\kappa(a)\in P_L$.
\item Let $a\in A(L)$ and let $g\in P_L$ be a function such that $g\preceq h_a$. Since $h_a(a)=0$, we have $g(a)=0$. Then, for any $x\in L$, 
	$$h_a(x)=\delta\left( x \vee a\right)\leq g\left( x\right) + g(a)= g \left( x \right).$$
So $h_a =g$ and therefore, $h_a$ is minimal in $P_L$. 

\item Let $a,b\in A(L)$ such that $a\neq b$. Then, $h_a(a)=0$, whereas $h_b(a)=\delta(a\vee b)> 0$. This means that $\kappa(a)\neq \kappa(b)$ and therefore, $\kappa$ is injective in $A(L)$.

\item Suppose that $\delta$ is strictly monotone, and let $a,b \in L$ such that $a\neq b$. If $a<b$, then $a<a\vee b=b$ and then, $h_b(a)=\delta(a\vee b)=\delta(b)>\delta(a)=h_a(a)$. The case $b<a$ is analogous. If $a$ and $b$ are not comparable, then $a\vee b > a$ and then $h_b(a)=\delta(a\vee b)>\delta(a)=h_a(a)$. We conclude that $\kappa$ is injective.
\end{enumerate}
  \end{proof}

For classical diversities \cite{Bryant2012}, the analogue of the map $\kappa$ is an embedding of the original diversity $(X,\delta)$ into its tight span, which is itself a diversity. The analogous property for lattice diversities would be if $\kappa$ were a lattice homomorphism.
Under the natural ordering of $P_L \cup \{\delta\}$, the join of two elements is the maximum and the meet of two elements is the minimum. So we would have that $h_{a \vee b}=\max(h_a,h_b)$ and $h_{a \wedge b} = \min(h_a,h_b)$.
 However, the following example shows that this is not the case in general.
Let $L$ be the set of all finite subsets of $\{a,b,c\}$. Define $\delta$ to take value $1$ on all two-element sets and $\delta(\{a,b,c\})=3/2$.  Then $(L,\delta)$ is a lattice diversity, but
$h_{\{a,b\}}(c) = \delta(\{a,b,c\})=3/2$ and 
$h_{\{a\}}(c)=h_{\{b\}}(c)=1$, 
so $h_{\{a\} \vee \{b\}} (c) \neq \max(h_{\{a\}}(c),h_{\{b\}}(c))$.

Now we give some examples of the construction of the tight span for lattice diversities. 

\begin{example}
If $(X,\delta)$ is a classical diversity, then the lattice diversity $(\Pf(X),\delta)$ has the same tight span as $(X,\delta)$ does in the theory in 
\cite{Bryant2012}. 
\end{example}

\begin{example}
Let $L$ be $\mathbf{M}_3$ the lattice given by the Hasse diagram:
\begin{figure}[H]
\centering
\includegraphics[height=30mm]{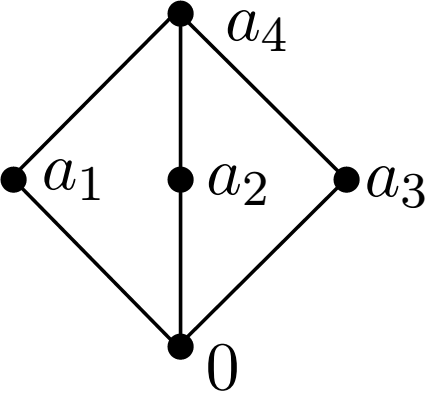}
\end{figure}
$\mathbf{M}_3$ is a modular but non-distributive lattice.

Any diversity function $\delta\colon L\rightarrow \Re_{\geq 0}$ is determined by the value $\alpha=\delta(a_4)$, since we have $\delta(0)=\delta(a_1)=\delta(a_2)=\delta(a_3)=0$. 
Let $f \colon L \rightarrow \mathbb{R}^+$.
$f$ is determined by the values $f_i=f(a_i)$ and $f_0=0$.
Proposition~\ref{prop:basic_properties_TP} implies that $f \in T_L$
 iff	
	$$\begin{array}{lll}
	f_0,f_1,f_2,f_3\geq 0, & f_4\geq \alpha, & f_1+f_2\geq \alpha,\\
	f_1+f_3\geq \alpha, & f_2+f_3\geq \alpha, & 
	\end{array}$$
and for each $f_i$ at least one of the inequalities it is in holds as an equality. 

Let $f \in T_L$. We immediately have $f_0=0$ and $f_4=\alpha$.
We do a case analysis depending on which inequalities hold as equalities. 
First if $f_1=0$, the remaining inequalities are $f_2 \geq \alpha$, $f_3 \geq \alpha$ and $f_2+f_3 \geq \alpha$. The first two imply that the third is strict and so the first two must hold as equalities. by looking at the remaining inequalities we have $f_2=f_3=\alpha$.
This gives the point $f=v_1=(0,0,\alpha,\alpha,\alpha)$.
 Likewise, if $f_2=0$ we have $f_1=f_3=\alpha$ and if $f_3=0$ we have $f_1=f_2=\alpha$. This gives us the points $f=v_2=(0,\alpha,0,\alpha,\alpha)$ and $f=v_1=(0,\alpha,\alpha,0,\alpha)$.
  These three points can be checked to satisfy the conditions for being in $T_L$. So now suppose $f_1,f_2,f_3 >0$. Then at least two out of the three inequalities $f_1+f_2\geq \alpha$, $f_1+f_3\geq \alpha$, $f_2+f_3\geq \alpha$ must hold as equalities. Suppose is is the first two: letting $f_1=x$ we get a set of points $f=(0,x,\alpha-x,\alpha-x,\alpha)$ which can be checked to be in the tight span if $x \in [0,\alpha/2]$. By choosing the other two pairs of inequalities we get further sets of points  $f=(0,\alpha-y,y,\alpha-y,\alpha)$ for $y \in [0,\alpha/2]$, and $f=(0,\alpha-z,\alpha-z,z,\alpha)$ for $z \in [0,\alpha/2]$, all of which can be checked to be in the tight span.
  
  All of this implies that 
  \begin{align*}
	T_L=&\{(0,f_1,\alpha-f_1,\alpha-f_1,\alpha)\mid 0\leq f_1\leq \alpha/2 \}\cup \\
		&\{(0,\alpha-f_2,f_2,\alpha-f_2,\alpha)\mid 0\leq f_2\leq \alpha/2 \} \cup \\
		&\{(0,\alpha-f_3,\alpha-f_3,f_3,\alpha)\mid 0\leq f_3\leq \alpha/2 \}
	\end{align*}	
	which we depict a projection of in the following figure.
\begin{figure}[H]
\centering
\includegraphics[height=30mm]{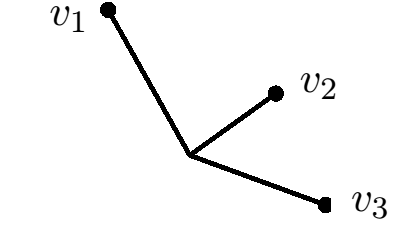}
\end{figure}

%
%
%

Using the definition of $\kappa$ we can determine the functions $\kappa(a_i)$:
	$$\begin{array}{ll} 
	\kappa(0) = (0,0,0,0,\alpha) & \kappa(a_1) = (0,0,\alpha,\alpha,\alpha) \\
	\kappa(a_2) = (0,\alpha,0,\alpha,\alpha) & \kappa(a_3) = (0,\alpha,\alpha,0,\alpha)\\
	\kappa(a_4) = (\alpha,\alpha,\alpha,\alpha,\alpha) & \\
	\end{array}$$
As we can see, $\kappa(a_i)$ corresponds to the point $v_i$ and therefore, the elements $a_1,a_2,a_3\in A(S)$ are mapped injectively to $v_1,v_2,v_3$ respectively. Also observe that the tight span has the same structure as the tight span of the induced metric on the atoms.
\end{example}

\begin{example}
Let $L$ be $\mathbf{N}_5$, the lattice given by the Hasse diagram
\begin{figure}[H]
\centering
\includegraphics[height=27mm]{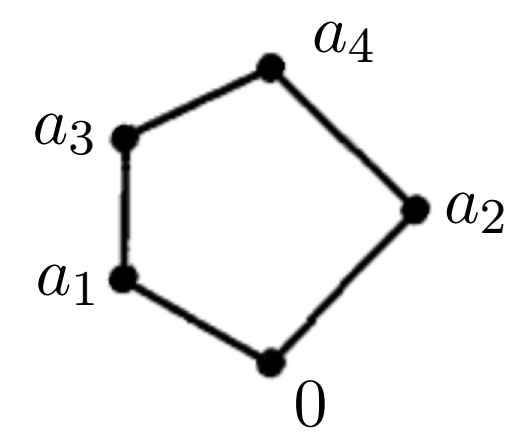}
\end{figure}

$\mathbf{N}_5$ is the simplest example of a non-modular lattice, and is hence non-distributive. Any diversity $\delta\colon L\rightarrow \Re_{\geq 0}$ is determined by the two values
	$$\alpha=\delta(a_3),\;\; \beta=\delta(a_4),$$
where $0<\alpha\leq \beta$. Let $f\colon L\rightarrow\Re_{\geq 0}$ be given by $f_i:=f(a_i)$. In order for $f\in P_L$, $f$ must satisfy condition (\ref{eq:defn_P}), which translates to the following inequalities:
	\begin{equation}\label{eq:example_tightspan_ineq}
	\begin{array}{lll}
	f_1,f_2\geq 0, & f_3\geq \alpha, & f_4\geq \beta, \\
	f_1+f_2\geq \beta, & f_2+f_3\geq \beta. &  \\
	\end{array}
	\end{equation}
By definition, it holds that $f\in T_L$ if and only if $f$ satisfies the inequalities (\ref{eq:example_tightspan_ineq}) and each $f_i$ is in one inequality that holds as an equality. It follows immediately that $f_4=\beta$.\\

Let $f \in T_L$.
Then $f_1,f_2\geq 0$ and $f_1+f_2\geq \beta$. If $f_1+f_2>\beta$, then $f_1=0$ (since $f_1 \geq 0$  is the only remaining inequality $f_1$ appears in) and then, $f_2>\beta$. Therefore, $f_2+f_3=\beta$, which contradicts $f_2>\beta$. So we must have $f_1+f_2=\beta$ for all $f\in T_L$.  Let $x=f_1$ and so $f_2=\beta-x$, which yields $f_3=\max(\alpha,x)$. Also, since $f_2\geq 0$, we have $f_1\leq \beta$. This gives the points $f=(0,x,\beta-x,\max(\alpha,x))$ for $x \in [0,\beta]$ as the tight span of $L$.

Denote $v_1=(0,0,\beta,\alpha,\beta)$, $v_2=(\beta,0,\beta)$ and $p=(\alpha,\beta-\alpha,\alpha)$.  Then $T_L$ is given by the union of segments $\overline{v_1,p}$ and $\overline{p,v_2}$, as depicted in the following figure:
\begin{figure}[H]
\centering
\includegraphics[height=23mm]{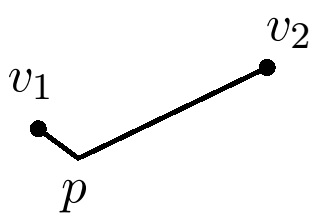}
\end{figure}

The functions $\kappa(a_i)$ are given by

	$$\begin{array}{ll} 
	\kappa(0) = (0,0,0,\alpha,\beta) & \kappa(a_1) = (0,0,\beta,\alpha,\beta) \\
	\kappa(a_2) = (0,\beta,0,\beta,\beta) & \kappa(a_3) = (\alpha,\alpha,\beta,\alpha,\beta)\\
	\kappa(a_4) = (\beta,\beta,\beta,\beta,\beta). & \\
	\end{array}$$

\end{example}

\section{Discussion} \label{sec:discussion}

Though many aspects of the theory of classical diversities introduced in \cite{Bryant2012} carry over directly to lattice diversities, we see that most of the theory of tight spans does not. As established in the previous section, the natural mapping from a lattice diversity to the set of real-valued functions on the lattice is not a lattice homomorphism. 
This motivates looking for other ways of defining the tight span of a lattice diversity.

In both the case of metrics \cite{dress1984} and classical diversities \cite{Bryant2012}, the tight span 
has a natural definition within the framework of category theory. If we view a metric space as an object in the concrete category {\bf{Met}} of metric spaces with non-expansive mappings, then the tight span of a metric space is its \emph{injective hull}, which is the unique essential embedding of the metric space into an injective object \cite{adamek2004}. The analogous result is true for classical diversities in the category {\bf{Div}} \cite{Bryant2012}. In general, for any concrete category (which we specify by its objects and morphisms) if all objects have injective hulls then the injective hulls are a good candidate for an analogue of the tight span. However, there are many interesting categories for which most objects do not have injective hulls. One example is the category of lattices {\bf{Lat}} where the only injective objects are single point lattices. This follows from the fact that any lattice with two or more points has arbitrarily large essential extensions \cite[Prop.\ 8]{bruns1968} and  that injective objects have no proper essential extension \cite[Prop.\ 9.15]{adamek2004}.

To pursue this avenue for lattice diversities, we start by defining a category where our objects are lattice diversities and define morphisms to be maps $f \colon (L_1,\delta_1) \rightarrow (L_2,\delta_2)$ between lattices that satisfy
\begin{enumerate}
\item $f$ is a lattice homomorphism,
\item $\delta_2(f(a))\leq \delta_1(a)$, for all $a \in L$.
\end{enumerate}
 Now a minimal requirement for objects in this category to have injective hulls is that each object be embeddable into some injective object. But what are the injective objects of this category? Unfortunately, they are the single point lattices with the trivial diversity on them.
 
 To see this,
 recall that an injective object $X$ in a concrete category is an object for which, if $A \stackrel{f}{\rightarrow} X$ is  a morphism and $A \stackrel{i}{\rightarrow} B$ is an embedding then there is a morphism $B \stackrel{g}{\rightarrow} X$ such that $g \circ i = f$.

\begin{prop}
In the concrete category of lattice diversities the only injective objects are single-point lattices with trivial diversities.
\end{prop}
\begin{proof} The proof follows nearly exactly the proof that there are no injective lattices with more than one point; see, for example, \cite[Prop.\ 8]{bruns1968}.
Let $(X,\delta)$ be an injective lattice diversity where $X$ is bounded and has two distinct points. 

$X$ has a 0 and at least one other point we denote as $e$. We can write $X$ as consisting of $0$, $e$, and three sets of points: $A_{(0,e)}$, which is all points between $0$ and $e$, $A^*$ which is all points not comparable to $e$, and $A_{(e,\infty)}$ which is all points strictly greater than $e$.

  Let $A = X$ and $f$ be the identity map, so that $(A,\delta)$ is also a lattice diversity.

Now we let $B$ be $A$ taken together with three new elements $a,b,c$. 
We define $a,b,c$ to be strictly greater than 0, strictly less than $e$, and not comparable to any points in $A_{(0,e)}$ or $A^*$ or to each other.
Since $a,b,c$ are atoms, we extend $\delta$ on $a,b,c$ to be 0, which makes $(B,\delta)$ a lattice diversity.  We claim that there is no lattice homomorphism from $B$ to $X$ that is the identity on $X=A$. This then implies that there is no 
morphism from $(B,\delta)$ to $(X,\delta)$ that is the identity on $A=X$.

To define a homomorphism from $B$ to $X$ that extends the identity, we need to specify $g(a), g(b), g(c)$. Starting with $a$, since $a \leq e$, $g(a)\leq g(e)=e$ and so $g(a) \in A_{(0,e)}$.
Suppose $g(a)=x \in A_{(0,e)}$. Then $g(x) \vee g(a) = x \vee x = x$, but $a \vee x= 1$. So $1= g(x \vee a) \neq g(x) \vee g(a)$. Hence each of $g(a), g(b), g(c)$ is either $0$ or $1$, and two of them must get mapped to the same one. So suppose $g(a)=g(b)=1$. This is a problem because $a\wedge b= 0$ but $g(a) \wedge g(b)=1$. Likewise if $g(a)=g(b)=0$. So no such homomorphism exists. Therefore $(X,\delta)$ is not injective.
%
  \end{proof}

Of course, there are other families of mappings between lattice diversities that we could define to be our morphisms, but we will encounter the same obstruction whenever the morphisms are required to be lattice homomorphisms.
   
 One further possibility is to consider restricted classes of lattice diversities, such distributive lattice diversities. The injective objects in the category of distributive lattices are the complete Boolean lattices \cite{bruns1968}, and so we do not meet the same obstruction that we do with general lattices. But we conjecture that the resulting theory will recapitulate that of the tight span theory for diversities, since Boolean lattices are isomorphic to the subalgebra of a power set by Stone's representation theorem \cite{davey2002introduction}.

\section*{Acknowledgements}
PT was supported by an NSERC Discovery Grant.
RF thanks CONACYT Project 45886 for support. MT thanks  CONACYT grant 223676 for support.

%
%

\bibliographystyle{plain}
\bibliography{LatticeDiversities}   

\end{document}